\documentclass[12pt,reqno,a4paper]{amsart}
\usepackage{amsmath,amssymb,amsthm,amsaddr}
\usepackage{enumitem}
\usepackage[margin=2.5cm,top=3.6cm,bottom=3.2cm,footskip=1cm,headsep=1.5cm]{geometry}
\usepackage[utf8]{inputenc}
\usepackage{caption}
\usepackage{mathtools}
\usepackage{float}
\usepackage{tikz}
\usetikzlibrary{matrix}

\usepackage{amsthm}
\usepackage{cleveref}
\usepackage{cancel}
\usepackage[british]{babel}

\author{H. Egger \and B. Radu}
\address{Department of Mathematics, TU Darmstadt, Germany}
\email{egger@mathematik.tu-darmstadt.de}
\email{radu@gsc.tu-darmstadt.de}

\title[Super-convergence and post-processing for the wave equation]{Super-convergence and post-processing for mixed finite element approximations of the wave equation}

\newtheorem{lemma}{Lemma}[section]
\newtheorem{problem}[lemma]{Problem}
\newtheorem{theorem}[lemma]{Theorem}

\theoremstyle{definition}
\newtheorem{remark}[lemma]{Remark}
\newtheorem*{example*}{Example}

\def\div{\mathrm{div}}
\def\grad{\nabla}

\def\tnorm{|\!|\!|}

\def\dt{\partial_t}
\def\dtt{\partial_{tt}}
\def\dttt{\partial_{ttt}}
\def\dtttt{\partial_{t}^{4}}

\def\u{u}
\def\p{p}

\def\eps{\epsilon}

\def\RR{\mathbb{R}}

\def\Th{\mathcal{T}_h}

\def\R{\mathbb{R}}
\def\dtau{\overline\partial_\tau}
\def\P{\text{P}}
\def\T{\mathcal{T}}
\def\BDM{\text{BDM}}

\numberwithin{equation}{section}
\numberwithin{table}{section}
\numberwithin{figure}{section}

\begin{document}

\begin{abstract} 
We consider the numerical approximation of acoustic wave propagation problems by mixed $\BDM_{k+1}$--$\P_k$ finite elements on unstructured meshes. Optimal convergence of the discrete velocity and super-convergence of the pressure by one order are established. 
Based on these results, we propose a post-processing strategy that allows us to construct an improved pressure approximation from the numerical solution. 
Corresponding results are well-known for mixed finite element approximations of elliptic problems and we extend these analyses here to the hyperbolic problem under consideration. 
We also consider the subsequent time discretization by the Crank-Nicolson method and 
show that the analysis and the post-processing strategy can be generalized to the fully discrete schemes. 
Our proofs do not rely on duality arguments or inverse inequalities and the results therefore apply also for non-convex domains and non-uniform meshes. 
\end{abstract}

\maketitle

\begin{quote}
\end{quote}

\begin{quote}
\noindent
{\small {\bf AMS-classification (2000):}
35L05, 35L50, 65L20, 65M60}
\end{quote}


\section{Introduction} \label{sec:intro}

The propagation of pressure waves of small amplitude in a Newtonian fluid or an elastic solid can be modeled by hyperbolic systems of the form
\begin{align}
a \dt \p + \div\, \u &= 0, \label{eq:sys1}\\
b \dt \u + \grad \p &= 0. \label{eq:sys2}
\end{align}
Here $\p$ and $\u$ denote the pressure and velocity fields, and the parameters $a$ and $b$ encode the material properties, e.g., the density of the medium or the speed of sound.
The two equations describe the conservation of mass and momentum 
and they can be derived under certain simplifying assumptions from the 
Euler equations of motion \cite{LandauLifshitz6}.

Finite element methods have been used quite successfully for the space discretization of wave propagation problems since many years. 
One common approach is to first rephrase the system \eqref{eq:sys1}--\eqref{eq:sys2} in its second order form
\begin{align}
a \dtt \p - \div ( b^{-1} \nabla \p) &= 0, \label{eq:sof}
\end{align}
which results from elimination of the velocity field $\u$ via the momentum equation.
This second order form can then be discretized in space by standard continuous piecewise polynomial finite elements. A sub-sequent time discretization by appropriate time-stepping schemes leads to efficient and reliable fully discrete methods with good stability and approximation properties; see \cite{Baker76,BakerBramble79,BangerthRannacher99,CohenJolyRobertsTordjman01,Dupont73} 
for such schemes and their analysis.

The standard Galerkin approximation of the wave equation \eqref{eq:sof} can also be interpreted as a particular mixed discretization of the first order system \eqref{eq:sys1}--\eqref{eq:sys2}, with emphasis on the approximation of the pressure variable $p$.
In many cases, however, the velocity field is the quantity of main interest.
Then alternative dual mixed finite element approximations may be more appropriate, 
where the approximation of the velocity $u$ is emphasized instead.
We refer to \cite{CowsarDupontWheeler96,Geveci88,JenkinsRiviereWheeler02,Karaa11,WangBathe97} for results in this direction and to \cite{Cohen02,Joly03} for a comprehensive survey on mixed finite element methods for wave propagation problems.

\medskip 

In this paper, we consider the discretization of the so-called velocity-pressure formulation \eqref{eq:sys1}--\eqref{eq:sys2} by mixed finite element schemes based on $\BDM_{k+1}$-$\P_{k}$ elements \cite{BrezziDouglasMarini85}. 
Related finite element approximations
have been considered in \cite{Geveci88}, and mixed methods for the equivalent displacement-stress formulation have been investigated in \cite{CowsarDupontWheeler96,JenkinsRiviereWheeler02}. 
From the analyses presented in these works, one can infer that the errors of the $\BDM_{k+1}$-$\P_{k}$ semi-discretization can be bounded by
\begin{align} \label{eq:ee}
\|p(t) - p_h(t)\|_{L^2(\Omega)} + \|u(t)-u_h(t)\|_{L^2(\Omega)}  \le C h^{k+1}, 
\end{align}
provided that the solution $(p,u)$ is sufficiently smooth. 
This estimate is optimal in view of the approximation properties of the pressure space $\P_{k}$, but sub-optimal regarding the velocity approximation in $\BDM_{k+1}$. 
Note that for related discretizations based on $\text{RT}_k$-$\P_k$ elements, the estimate 
is of optimal order in both variables. 

As a first result of our paper, we will prove that also for the $\BDM_{k+1}$-$\P_k$ approximation, the error is in fact of optimal order, more precicely, 
\begin{align} \label{eq:ee2}
\|\pi_h p(t) - p_h(t)\|_{L^2(\Omega)} + \|u(t)-u_h(t)\|_{L^2(\Omega)}  \le C h^{k+2}.
\end{align}
Here $\pi_h$ denotes the $L^2$-projection onto the discrete pressure space. 
This estimate also shows that the projected pressure error even exhibits super-convergence by one order. 
The overall approximation therefore is balanced, although different polynomial orders for the approximation of $p$ and $u$ are used.
Similar super-convergence results are well-known for elliptic problems, 
see e.g. \cite{BoffiBrezziFortin13,BDDF87,BrezziDouglasMarini85} and \cite[Ch.~1]{BoffiEtAl08} for a comprehensive overview. 
Here we extend these results to the hyperbolic problem under consideration.

Let us put this first result into perspective of previous work:
In \cite{Chen99}, an improved error estimate $\|\pi_h p(t) -p_h(t)\|\le C h^{2}$ for the projected pressure error 
has been established for the $\text{RT}_0$-$\P_0$ discretization of the wave equation. 
The analysis there however assumes the underlying elliptic problem to be $H^2$-regular, which we do not require here.
In addition, only first order convergence
$\|u(t)-u_h(t)\| \le C h$ for the velocity can be obtained 
due to the inferior approximation of the velocities in $\text{RT}_0$. 
Moreover, the extension of the analysis in \cite{Chen99} to the time discrete problem does  not seem straight-forward. 

Optimal order estimates $\|u(t)-u_h(t)\|\le C h^{k+2}$ have been established in \cite{Monk92} for a second order formulation of Maxwell's equations; see also \cite{MakridakisMonk95} for the subsequent time discretization. 
Although the second solution component is not present in this formulation,
the estimates of these works can be seen to be related to \eqref{eq:ee2} in two dimension.

Various super-convergence results for approximations of wave phenomena are available for structured or almost structured grids; see e.g. \cite{Chen96,ChenHuang98,DouglasDupontWheeler78,WangChenTang13}. 
In contrast to these approaches, which are usually based on duality arguments and require regularity of the underlying meshes, our arguments here are based only on certain compatibility conditions of the approximation spaces, and our results therefore apply to more general situations.

\medskip 

Based on sharp estimates for the projected pressure error, different post-processing strategies have been proposed for mixed finite element approximations of elliptic problems; see e.g. \cite{ArnoldBrezzi85,BrezziDouglasMarini85,Stenberg91}.
These aim at constructing improved approximations $\widetilde p_h$ for the pressure from  the discrete solution $(p_h,u_h)$ and the problem data by local computations.
The second contribution of our manuscript will be to extend these results to the hyperbolic problem under consideration. For $t \ge 0$ we construct an improved pressure approximation $\widetilde p_h(t) \in \P_{k+1}$, such that
\begin{align} \label{eq:ee3}
\|u(t)-u_h(t)\|_{L^2(\Omega)} + \|p(t) - \widetilde p_h(t)\|_{L^2(\Omega)} \le C h^{k+2}.
\end{align}
Another post-processing strategy has been proposed in \cite{Chen99} for the discretization of the wave equation by $\text{RT}_0$-$\P_0$ elements. An improved pressure approximation $\widetilde p_h(t)$ is constructed there on a macro mesh and the estimate $\|p(t) - \widetilde p_h(t)\| \le C h^{2}$ is established. 
Apart from the availability of a macro mesh, the proof of this estimate again uses duality arguments and requires $H^2$-regularity of the underlying elliptic problem.
We do not require such assumptions and our results therefore hold for more general domains and unstructured as well as locally refined meshes. 

As a third step of our analysis, we finally also investigate the subsequent time discretization. For ease of presentation, we only consider the Cranck-Nicolson method in detail, but the arguments can be extended to higher order schemes. The super-convergence estimates for the projected pressure error as well as the post-processing strategy are extended to the resulting fully discrete schemes, which give approximations with optimal convergence order in space and time.

\bigskip

The remainder of the manuscript is organized as follows:
In Section~\ref{sec:prelim}, we give a complete description of the problem under investigation and recall some preliminary results. 
Section~\ref{sec:semi} then presents the a-priori error analysis for the semi-discretization. 
The post-processing strategy for the semi-discretization is introduced and analyzed in Section~\ref{sec:post}.
In Section~\ref{sec:full}, we consider the subsequent time discretization by the Crank-Nicolson method and derive order optimal estimates for the velocity and the projected pressure error. 
The construction and the analysis of the post-processing strategy is extended to the fully discrete schemes in Section~\ref{sec:fullpost}.
Some numerical tests are presented in Section~\ref{sec:num}, for illustration of the theoretical results. We conclude with a short summary and a discussion of some open topics for future research.

\section{Preliminaries} \label{sec:prelim}
Let $\Omega\subseteq\R^d$ be a bounded polyhedral Lipschitz domain in dimension $d=2$ or $d=3$, and let $T>0$ denote a finite time horizon. 
We consider the first order hyperbolic system 
\begin{alignat}{3}
a \dt\p + \div\, \u &= f \qquad \text{in }\Omega\times(0,T)\label{eq:sysm1},\\ 
 b \dt\u + \grad \p &= g \qquad \text{in }\Omega\times(0,T)\label{eq:sysm2},
\intertext{with homogeneous boundary conditions}
   p&=0    \qquad \text{on }\partial\Omega\times(0,T),\label{eq:sysm3}
\intertext{and initial values prescribed by}
 p(0)=p_0, \quad u(0)&=u_0 \quad \ \; \text{in }\Omega .\label{eq:sysm4}
\end{alignat}
The coefficient functions $a,b : \Omega \to \RR$ are assumed to be bounded and uniformly positive throughout. 
Using standard arguments of semi-group theory \cite{Evans98,Pazy83}, one can then show 
\begin{lemma}[Classical solution]
For any $u_0 \in H(\div;\Omega)$ and $p_0 \in H_0^1(\Omega)$ and any $f \in W^{1,1}(0,T;L^2(\Omega))$ and $g \in W^{1,1}(0,T;L^2(\Omega)^d)$  there exists a unique classical solution 
\begin{align*}
(p,u)\in C([0,T],H^1_0(\Omega)\times H(\div;\Omega))\times C^1([0,T],L^2(\Omega)\times L^2(\Omega)^d)
\end{align*}
of the system \eqref{eq:sysm1}--\eqref{eq:sysm4} and its norm can be bounded by the norm of the data. 
\end{lemma}
Here $L^2(\Omega)$, $H^1(\Omega)$, and $H(\div,\Omega)=\{u\in L^2(\Omega)^d\,|\, \div\, u \in L^2(\Omega)\}$ are the usual Lebesgue and Sobolev spaces. The functions in $H_0^1(\Omega)$ additionally have zero trace on the boundary. 
The spaces $C^k([0,T];X)$ and $W^{k,p}(0,T;X)$ consist of functions of time with values in a Hilbert space $X$; we refer to \cite{Evans98} for details and further notation.

\begin{remark}\label{remark:higherregularity}
By formal differentiation of the system with respect to time, one can also obtain existence and uniform bounds for the time derivatives $(\dt^k p,\dt^k u)$ of higher order, provided that the usual compatibility conditions between initial and boundary conditions are satisfied. 
Higher regularity with respect to the space variable can then be obtained via elliptic regularity estimates; we again refer to \cite{Evans98} for details.
\end{remark}

The mixed finite element discretizations investigated in the subsequent sections 
are based on the following variational characterization of classical solutions.

\begin{lemma}[Variational characterization] $ $\\
Let $(p,u)$ denote a classical solution of \eqref{eq:sysm1}--\eqref{eq:sysm4}. 
Then for all $0 \le t \le T$ there holds
\begin{alignat}{5}
(a \dt p(t),q)_\Omega+(\div\, u(t),q)_\Omega &=(f(t),q)_\Omega \qquad \forall q\in L^2(\Omega), \label{eq:var1} \\ %
(b \dt u(t),v)_\Omega-(p(t),\div \,v)_\Omega &=(g(t),v)_\Omega \qquad \forall v\in H(\div,\Omega). \label{eq:var2}
\end{alignat}
\end{lemma}
\noindent 
The symbol $(p,q)_\Omega=\int_\Omega p q \; dx$ is used here to denote the scalar product of $L^2(\Omega)$. 
\begin{proof}
The assertion follows by testing the equations \eqref{eq:sysm1}--\eqref{eq:sysm2} with appropriate functions, integration-by-parts in the second equation, and use of the boundary conditions.
\end{proof}

\section{Semi-discretization} \label{sec:semi}

\subsection{Preliminaries}
Let $\Th = \{K\}$ be a conforming simplicial mesh of the domain $\Omega$. 
We denote by $h_K$ and $\rho_K$ the diameter and inner circle radius of the element $K$, and call $h=\max_K h_K$ the mesh size.
We assume that $\Th$ is $\gamma$-shape regular, i.e., that $\gamma h_K \le \rho_K \le h_K$ holds for all $K \in \Th$ with some uniform constant $\gamma>0$. 
We further denote by $\P_k(K)$ the space of polynomial functions on $K$ of degree less or equal to $k$.
In our estimates we repeatedly use constants $C$ which may have different values at different occasions, and which may depend on the polynomial degree $k$, on the mesh regularity constant $\gamma$, and on the bounds for the parameters $a$ and $b$. 
These constants will however always be independent of the mesh size $h$. 

As finite dimensional approximations for the function spaces $L^2(\Omega)$ and $H(\div;\Omega)$ arising in the above variational characterization, we consider the finite element spaces
\begin{align}
Q_h
&= \{ q \in L^2(\Omega) : q|_K \in \P_k(K) \ \forall K \in \Th\} =: \P_k(\T_h), \label{eq:Qh}\\
V_h 
&=\P_{k+1}(\T_h)^d \cap H(\div;\Omega)
=:\BDM_{k+1}(\T_h) \cap H(\div;\Omega).
\end{align}
Note that functions in $v_h \in V_h$ have continuous normal components across element interfaces, while functions $q_h \in Q_h$ may be completely discontinuous across interfaces in general. Also recall \cite{BoffiBrezziFortin13,BrezziDouglasMarini85} that for simplicial elements $K$, we have $\P_k(K)^d = \BDM_{k}(K)$,
which explains our notation. 
The following well-known approximation results will be used several times.
\begin{lemma}[Projection operators]\label{lem:projections}
Let $Q_h$ and $V_h$ be defined as above. Then 
\begin{align} \label{eq:spaces}
\div V_h = Q_h,
\end{align}
and there exists a projection operator $\rho_h : H^1(\Th)^d \cap H(\div,\Omega) \to V_h$ such that 
\begin{alignat}{2} \label{eq:commuting}
&\div\, \rho_h u = \pi_h \div\, u \qquad && \forall u \in H^1(\Th)^d \cap H(\div,\Omega). 
\end{alignat}
Here $\pi_h : L^2(\Omega) \to Q_h$ is the $L^2$-orthogonal projection onto $Q_h$ defined by 
\begin{align} \label{eq:pQhK}
(\pi_h p, q_h)_\Omega = (p,q_h)_\Omega \qquad \forall q_h \in Q_h.
\end{align}
Moreover, the following approximation error estimates hold true:
\begin{alignat}{2}
&\|p-\pi_h p\|_{L^2(\Omega)}\leq C h^r |p|_{H^r(\Th)}\quad  && \forall p\in H^r(\Th), \ 0\leq r\leq k+1, \label{eq:approx1}\\
&\|u-\rho_h u\|_{L^2(\Omega)}\leq C h^r |u|_{H^r(\Th)}\quad  && \forall u\in H^r(\Th)^d \cap H(\div,\Omega), 1\leq r\leq k+2, \label{eq:approx2}
\end{alignat}
and $C$ only depends on the shape regularity constant $\gamma$ and the polynomial degree $k$.
\end{lemma}
We denote here by $H^k(\Th) = \{ q \in L^2(\Omega) : q|_K \in H^k(K)\}$ the broken Sobolev space of piecewise smooth functions, in analogy to our notation for piecewise polynomial spaces.
\begin{remark}
The relation \eqref{eq:commuting} is known as \emph{commuting diagram property} and will play an important role in our analysis. The projection operator $\rho_h$ can be defined locally for every element, which is why some additional regularity is required for the definition of $\rho_h$; see e.g. \cite[Chapter~2.5.2]{BoffiBrezziFortin13}. This artificial smoothness requirements could be relaxed, if a commuting quasi-interpolation operator \cite{Schoeberl01} would be used instead. 
\end{remark}

\subsection{Mixed finite element approximation}

For the discretization of the initial boundary value problem \eqref{eq:sysm1}--\eqref{eq:sysm4} in space, we consider the following mixed finite element approximation of the variational principle stated above.

\begin{problem}[Galerkin semi-discretization]\label{problem:semidiscrete} $ $\\
Find $(p_h,u_h) \in H^1(0,T;Q_h \times V_h)$ with 
\begin{alignat}{2}
p_h(0)=\pi_h p_0 \qquad \text{and} \qquad u_h(0)=\rho_h u_0,
\end{alignat} 
and such that for all $0 \le t \le T$ there holds
\begin{alignat}{2}
(a \dt p_h(t),q_h)_\Omega+(\div\, u_h(t),q_h)_\Omega &=(f(t),q_h)_\Omega \qquad \forall q_h\in Q_h, \label{eq:sysvd1} \\
(b \dt u_h(t),v_h)_\Omega-(p_h(t),\div\, v_h)_\Omega &=(g(t),v_h)_\Omega \qquad \forall v_h\in V_h \label{eq:sysvd2}.
\end{alignat}
\end{problem}

Existence and uniqueness of a semi-discrete solution $(p_h,u_h)$ follows immediately from the Picard-Lindelöf theorem. 
Moreover, the following discrete stability estimate holds. 
\begin{lemma}[Discrete energy estimate] $ $\\
Let $(p_h,u_h)$ denote the solution of Problem~\ref{problem:semidiscrete}.
Then 
\begin{align} \label{eq:discrete_energy}
&\|p_h(t)\|_{L^2(\Omega)} + \|u_h(t)\|_{L^2(\Omega)}\\
& \le C \big (\|p_h(0)\|_{L^2(\Omega)} + \|u_h(0)\|_{L^2(\Omega)} 
+ \int_0^t \|f(s)\|_{L^2(\Omega)} + \|g(s)\|_{L^2(\Omega)} ds \big) \notag
\end{align}
with constant $C$ only depending on the bounds for $a$ and $b$ and the time horizon $T$.
\end{lemma}
\begin{proof}
The estimate follows by testing \eqref{eq:sysvd1}--\eqref{eq:sysvd2} 
with $q_h=p_h(t)$, $v_h=u_h(t)$ and application of Cauchy-Schwarz and G{\aa}rding inequalities; 
see e.g. \cite[Theorem 4.1]{JenkinsRiviereWheeler02}.
\end{proof}

\subsection{Error estimation}
We next turn to the derivation of a-priori error estimates. 
Using the projection operators defined above, we split the error in the usual way by
\begin{align} \label{eq:errorsplitting}
\|p(t)-p_h(t)\|_{L^2(\Omega)}  &+ \|u(t)-u_h(t)\|_{L^2(\Omega)} \\
&\le \|p(t)-\pi_h p(t)\|_{L^2(\Omega)} + \|u(t)-\rho_h u(t)\|_{L^2(\Omega)} \notag \\
& \qquad \qquad + \|\pi_h p(t)-p_h(t)\|_{L^2(\Omega)} + \|\rho_h u(t)-u_h(t)\|_{L^2(\Omega)} \notag
\end{align}
into a \emph{projection error} and a \emph{discrete error} component.
The projection error can be bounded readily by Lemma~\ref{lem:projections}. 
For the discrete error component, we use
\begin{lemma}[Discrete error]
Let $(p,u)$ be sufficiently smooth. Then
\begin{align} \label{eq:discrete_error}
&\|\pi_h p(t) - p_h(t)\|_{L^2(\Omega)} + \|\rho_h u(t) - u_h(t)\|_{L^2(\Omega)} \\
&\le C \int_0^t \|b\|_{L^\infty(\Omega)} \|\dt u(s) - \rho_h\dt u(s)\|_{L^2(\Omega)} + \|a-\pi_h^0 a\|_{L^\infty(\Omega)} \|\dt p(s) - \pi_h  \dt p(s)\|_{L^2(\Omega)} ds, \notag
\end{align}
with the same constant $C$ as in the discrete energy estimate \eqref{eq:discrete_energy} above.
\end{lemma}
\noindent
Here $\pi_h^0: L^2(\Omega) \to P_0(\Th)$ denotes the $L^2$-orthogonal projection onto piecewise constants. 
\begin{proof}
Define $r_h(t) = \pi_h p(t) - p_h(t)$ and $w_h(t) = \rho_h u(t) - u_h(t)$. 
Using \eqref{eq:var1}--\eqref{eq:var2} and  \eqref{eq:sysvd1}--\eqref{eq:sysvd2}, 
we then see that $(r_h,w_h)$ satisfies 
$r_h(0)=0$ and $w_h(0)=0$, as well as 
\begin{alignat}{5}
(a \dt r_h(t),q_h)_\Omega+(\div\ w_h(t),q_h)_\Omega &= (f_h(t),q_h)_\Omega \label{eq:errh1}\\
(b \dt w_h(t),v_h)_\Omega-(r_h(t),\div\, v_h)_\Omega &= (g_h(t),v_h)_\Omega \label{eq:errh2}
\end{alignat}
for all $q_h \in Q_h$ and $v_h \in V_h$, and for all $0 \le t \le T$ with 
\begin{align*}
(f_h(t),q_h)_\Omega  &=(a (\dt \pi_h p(t) - \dt p(t)),q_h)_\Omega + (\div (\rho_h u(t) - u(t)),q_h)_\Omega  = :(i)+(ii)\\
(g_h(t),v_h)_\Omega  &=(b (\dt \rho_h u(t) - \dt u(t)),v_h)_\Omega - (\pi_h p(t) - p(t), \div\, v_h)_\Omega = :(iii)+(iv).
\end{align*}
Because of the compatibility condition \eqref{eq:spaces} and the commuting diagram property \eqref{eq:commuting}, the terms (ii) and (iv) vanish. 
The first term can be further expanded as 
\begin{align*}
(i) &= ( (a - \pi^0_h a) \; (\pi_h \dt p(t) - \dt p(t)), q_h)_\Omega + (\pi_h^0 a \; (\pi_h \dt p(t) - \dt p(t)),q_h)_\Omega.
\end{align*}
The last term in this expression again vanishes due to orthogonality of the projection $\pi_h$.
The discrete error $(r_h,w_h)$ thus satisfies a discrete variational problem 
with right hand sides $f_h(t)=(a-\pi_h^0 a) (\pi_h \dt p(t) - \dt p(t))$ and $g_h(t)=b (\rho_h \dt u(t) - \dt u(t))$ and zero initial conditions. The assertion now follows from the discrete energy estimate \eqref{eq:discrete_energy}.
\end{proof}

A combination of the previous arguments already yields our first main result.
\begin{theorem}[Error estimate for the semi-discretization] \label{theorem:semidiscrete} $ $\\
Assume that $a \in W^{1,\infty}(\Th)$ and that $(p,u)$ is sufficiently smooth. 
Then 
\begin{align*}
\|\pi_h p(t) - p_h(t)\|_{L^2(\Omega)}  + \|u(t) - u_h(t)\|_{L^2(\Omega)} 
\le C h^{k+2} C_1(p,u,t),
\end{align*}
with constant $C_1(p,u,t)=\|u(t)\|_{H^{k+2}(\Omega)} + \int_0^t \|\dt u(s)\|_{H^{k+2}(\Omega)} +\|\dt p(s)\|_{H^{k+1}(\Omega)} ds$. 
\end{theorem}

\begin{remark}
As can be seen from the proof, order optimal rates for less regular solutions can be obtained by simply replacing $k+2$ in all terms by an index $r\le k+2$. 
The assumption $a \in  W^{1,\infty}(\Th)$ for the parameter can be dropped for $r \le k+1$. 
\end{remark}

\subsection{Auxiliary results}
For the analysis of our post-processing strategy presented in the next section, we also require an error estimate for the time derivative of the velocity.
\begin{lemma}[Error estimates for the time derivatives] \label{lem:dtu} $ $\\
Let $(p,u)$ be sufficiently smooth. Then 
\begin{align*}
\|\dt u(t) - \dt u_h(t)\|_{L^2(\Omega)} 
\le C h^{k+1} C_2(p,u,t),
\end{align*}
with constant $C_2$ defined by $C_2(p,u,t) = \|\dt u(t)\|_{H^{k+1}(\Omega)} + \|\dt u(0)\|_{H^{k+1}(\Omega)} + \|\dt p(0)\|_{H^{k+1}(\Omega)}$ \\$+\int_0^t \|\dtt u(s)\|_{H^{k+1}(\Omega)} + \|\dtt p(s)\|_{H^{k+1}(\Omega)} ds$.
\end{lemma}
\begin{proof}
We again split the error into a projection and a discrete component
\begin{align*}
\|\dt u(t) - \dt u_h(t)\|_{L^2(\Omega)} \le  \|\dt u(t) - \rho_h \dt u(t)\|_{L^2(\Omega)} + \|\dt \rho_h u(t) - \dt u_h(t)\|_{L^2(\Omega)}.
\end{align*}
The projection error can be bounded by Lemma~\ref{lem:projections}. 
Now define $r_h(t) = \dt \pi_h p(t) - \dt p_h(t)$ and $w_h(t) = \dt \rho_h u(t) - \dt u_h(t)$. 
By formal differentiation of \eqref{eq:errh1}--\eqref{eq:errh2}, one can see that $(r_h,w_h)$
again solves a discrete variational problem of the form \eqref{eq:errh1}--\eqref{eq:errh2}
with $f_h(t)=(a-\pi_h^0 a) (\pi_h \dtt p(t) - \dtt p(t))$ and $g_h(t)=b(\rho_h \dtt u(t) - \dtt u(t))$, 
and initial conditions 
\begin{align*}
r_h(0) = \pi_h \dt p(0) - \dt p_h(0) 
\qquad \text{and} \qquad 
w_h(0) = \rho_h \dt u(0) - \dt u_h(0).
\end{align*}
To take advantage of the discrete energy estimate, we have to derive bounds for these initial values. 
We start with an auxiliary result: 
Given $\hat q_h \in Q_h$, we define $q_h \in Q_h$ by 
\begin{align*}
(a q_h, r_h)_\Omega &= (\hat q_h,r_h)_\Omega \qquad \forall r_h \in Q_h. 
\end{align*}
Since $a$ is bounded from above and uniformly positive, this elliptic problem is 
uniquely solvable and we have $c_1 \|\hat q_h\|_K \le \|q_h\|_K \le c_2 \|\hat q_h\|_K$.
Using this construction and the discrete variational problem, 
we can see that for any $\hat q_h \in Q_h$ there holds
\begin{align*}
(\dt p_h(0), \hat q_h)_\Omega 
&=(a \dt p_h(0), q_h)_\Omega 
 =(f(0) - \div\, u_h(0), q_h)_\Omega \\
&=(f(0) - \div\, u(0),q_h)_\Omega 
 =(a \dt p(0),q_h) \\
&=(\pi_h \dt p(0), \hat q_h)_\Omega + (a (\dt p(0) - \pi_h \dt p(0)), q_h)_\Omega.
\end{align*}
Thus $\dt p_h(0) = \pi_h \dt p(0) + e^p_h$
with $(e^p_h,\hat q_h)_\Omega = (a (\dt p(0) - \pi_h \dt p(0)), q_h)_\Omega$.
Using the relation between $q_h$ and $\hat q_h$, this allows us to estimate 
\begin{align*}
\|\dt p_h(0) - \pi_h \dt p(0)\|_{L^2(\Omega)}
&\le c_2 \|a\|_{L^\infty(\Omega)} \|\dt p(0) - \pi_h \dt p(0)\|_{L^2(\Omega)}.
\end{align*}
In a similar manner, one can also show that 
\begin{align*}
\|\dt u_h(0) - \rho_h \dt u(0)\|_{L^2(\Omega)}
&\le c_3 \|b\|_{L^\infty(\Omega)} \|\dt u(0) - \rho_h \dt u(0)\|_{L^2(\Omega)}. 
\end{align*}
The assertion of the theorem now follows by 
using these estimates for the error in the initial condition, the discrete energy estimate, and the estimates for the projection errors.
\end{proof}

\begin{remark}
Similar as before, one can also formulate error estimates for orders $r \le k+2$ under appropriate smoothness assumptions on the solution. 
We will however only use the estimate for $r=k+1$ explicitly below.
\end{remark}

\section{Post-processing for the semi-discretization} \label{sec:post}

\subsection{Post-processing strategy}

The approximation properties of the finite element spaces $V_h$ and $Q_h$ for velocity and pressure are somewhat unbalanced. 
Due to the super-convergence of the projected pressure error, it is however 
possible to improve the pressure approximation by one order through local post-processing.
The starting point for our considerations is the following observation: 
Let us multiply equation \eqref{eq:sysm2} with some function $v=\nabla q$ and integrate over an element $K \in \Th$.
Then 
\begin{align}
(\grad\p,\grad q)_K = -(b \dt\u,\grad q)_K + (g,\grad q)_K.
\end{align}
The pressure $p$ can thus be found, up to an additive constant, by solution of an elliptic boundary value problem, once the velocity $u$ is known. 
This motivates 
\begin{problem}[Local post-processing strategy] $ $\\
For every $0 \le t \le T$, find $\widetilde p_h(t)\in \P_{k+1}(\Th)$ such that for all $K \in \Th$
\begin{alignat}{2}
(\grad\widetilde\p_h(t),\grad \widetilde q_h)_K &=(g(t),\grad \widetilde q_h)_K -(b \dt\u_h(t),\grad \widetilde q_h)_K \quad 
&&\forall \widetilde q_h\in \P_{k+1}(K), \label{eq:syspp1} \\ 
(\widetilde p_h(t), q_h^0)_K &= (p_h(t),q_h^0)_K &&\forall q_h^0 \in \P_0(K). \label{eq:syspp2}
\end{alignat}
\end{problem}
\noindent
The second condition could be written as $\pi_h^0 \widetilde p_h(t) = \pi_h^0 p_h(t)$ and fixes the constant.
\begin{remark}
The above procedure is inspired by the post-processing strategy 
for mixed finite element approximations of the Poisson problem presented in \cite{Stenberg91}. 
Let us emphasize that $\widetilde p_h(t)$ can be computed locally and independently for every element $K \in \Th$. 
 \end{remark}

\subsection{Error estimates}

Proceeding with similar arguments as in \cite{Stenberg91}, we now obtain

\begin{theorem}[Error estimate for the post-processing]\label{theorem:semidiscrete-pp} $ $\\
Under the assumptions of Theorem~\ref{theorem:semidiscrete}, we have 
\begin{align*}
\|p(t)-\widetilde p_h(t)\|_{L^2(\Omega)}
\leq C h^{k+2} \big( 
\|p(t)\|_{H^{k+2}(\Omega)} 
+ C_1(p,u,t) + C_2(p,u,t) \big),
\end{align*}
with constants $C_1(p,u,t)$ and $C_2(p,u,t)$ defined as in Theorem~\ref{theorem:semidiscrete} and Lemma~\ref{lem:dtu}. 
\end{theorem}
\noindent
Recall that by convention the constant $C$ is assumed independent of $h$, $u$, $p$, and $t$.
\begin{proof}
Let $\widetilde \pi_h: L^2(\Omega) \to \P_{k+1}(\Th)$ denote the $L^2$-orthogonal projection onto $\P_{k+1}(\Th)$. 
For ease of notation, we write $p=p(t)$ and omit the time argument in the following.
We can then split the error at time $t$ and on every element $K$ into
\begin{align} \label{eq:split}
\|p-\widetilde p_h\|_{L^2(K)} 
\leq\|p&-\widetilde \pi_h p\|_{L^2(K)} + \|\pi_h^0 (\widetilde \pi_h p- \widetilde p_h)\|_{L^2(K)} \notag \\
&+\|(\text{id}-\pi_h^0)(\widetilde \pi_h p-\widetilde p_h)\|_{L^2(K)}.
\end{align}
The first term is already of optimal order, and the second can be estimated as
\begin{align*}
\|\pi_h^0 (\widetilde \pi_h p - \widetilde p_h)\|_{L^2(K)} 
= \|\pi_h^0 (\pi_h p - p_h)\|_{L^2(K)}
\leq \|\pi_h p-p_h\|_{L^2(K)}.
\end{align*}
Here we used the orthogonality of the $L^2$-projections and \eqref{eq:syspp2} in the first step.
After summing over all elements, the second term in the error splitting can thus be bounded by the estimates of Theorem~\ref{theorem:semidiscrete}.
Now consider the third term $\widetilde q_h = (\text{id}-\pi_h^0)(\widetilde \pi_h p-\widetilde p_h)$. 
By the Poincaré-Wirtinger inequality for the convex domain $K$, we have
\begin{align*}
\|\widetilde q_h\|_{L^2(K)}\leq h_K \|\grad \widetilde q_h\|_{L^2(K)}.
\end{align*}
The gradient of $\widetilde q_h$ can be further estimated by 
\begin{align*}
\|\grad \widetilde q_h\|_{L^2(K)}^2
&=(\grad( \widetilde \pi_h p-\widetilde p_h),\grad \widetilde q_h)_K\\
&\leq(\grad(\widetilde \pi_h p-p),\grad \widetilde q_h)_K+(\grad(p-\widetilde p_h),\grad \widetilde q_h)_K\\
&=(\grad(\widetilde \pi_h p-p),\grad \widetilde q_h)_K-(b \dt(u-u_h),\grad \widetilde q_h)_K\\
&\leq \left(\|\grad(\widetilde \pi_h p-p)\|_{L^2(K)} + c \|\dt(u-u_h)\|_{L^2(K)}\right)\|\grad \widetilde q_h\|_{L^2(K)}.
\end{align*}
The third term in the error splitting can thus be bounded by 
\begin{align*}
\|(\text{id}-\pi_h^0)(\widetilde \pi_h p-\widetilde p_h)\|_{L^2(K)}
&\le h_K \|\grad(\widetilde \pi_h p-p)\|_{L^2(K)} + C h_K \|\dt(u-u_h)\|_{L^2(K)}.
\end{align*}
The assertion now follows by summing over all elements, 
using the estimates of Theorem~\ref{theorem:semidiscrete} and Lemma~\ref{lem:dtu}, 
and noting that $\|\nabla (\widetilde \pi_h p - p)\|_{L^2(K)} \le C h_K^{k+1} \|p\|_{H^{k+2}(K)}$.
\end{proof}

\begin{remark}
A different post-processing strategy was investigated in \cite{Chen99},
under the assumption that $\Th$ results from a uniform refinement of a macro-mesh $\T_H$.
An improved approximation $\widetilde p_H \in \P_{k+1}(\T_H)$ was obtained there for polynomial degree $k=0$ by local averaging of the finite element solution $p_h$ on macro elements $\widetilde K \in \T_H$.
This procedure may seem simpler at first sight, but it yields approximations $\widetilde p_H$ 
on a coarser mesh and additionally requires the availability of an appropriate macro mesh.
The post-processing strategy presented above does not rely on such restrictive assumptions.
\end{remark}

\section{Time discretization} \label{sec:full}

\subsection{The fully discrete scheme}
We now consider the time discretization of the mixed finite element scheme. For ease of presentation, we only consider the Crank-Nicolson method here. Similar results could however also be obtained for other methods.

Let us fix some $N \ge 1$ and set $\tau = T/N$. Then define $t^n = n \tau$ for $0 \le n \le N$. 
Given a sequence $\{p^n\}_{n \ge 0}$, we denote by
\begin{align}
\dtau p^{n-1/2} = \frac{p^{n}-p^{n-1}}{\tau} 
\qquad \text{and} \qquad  
p^{n-1/2} = \frac{p^{n}+p^{n-1}}{2}
\end{align}
the differences and averages at intermediate time steps. 
For an arbitrary function $p$ we first set $p^{n}=p(t^n)$ and 
then define the above symbols accordingly for the sequence $\{p^n\}_{n \ge 0}$.
As a fully discrete approximation to our problem, we now consider
\begin{problem}[Fully discrete problem]\label{problem:fullydiscrete} $ $\\
Set $p_h^0=\pi_h p_0$ and $u_h^0=\rho_h u_0$.
Then for $1 \le n \le N$  find $(p_h^n,u_h^n) \in Q_h \times V_h$ such that 
\begin{alignat}{2}
(a \dtau p_h^{n-1/2},q_h)_\Omega+(\div\, u_h^{n-1/2} ,q_h)_\Omega &=(f^{n-1/2},q_h)_\Omega \qquad &&\forall q_h\in Q_h, \label{eq:syscvfd1} \\ 
(b \dtau u_h^{n-1/2},v_h)_\Omega - (p_h^{n-1/2},\div\, v_h)_\Omega &=(g^{n-1/2},v_h)_\Omega &&\forall v_h\in V_h. \label{eq:syscvfd2}
\end{alignat}
\end{problem}
Since the spaces $Q_h$, $V_h$ are finite dimensional, it is not difficult to see that this iteration is well-defined. 
Moreover, the following stability estimates hold true.
\begin{lemma}[Energy estimate for the fully discrete scheme] \label{lem:fullydiscrete}$ $\\
Let $(p_h^n,u_h^n)$ for $n \ge 0$, be defined as above. Then
\begin{align*}
\|p_h^n\|_{L^2(\Omega)} &+ \|u_h^n\|_{L^2(\Omega)} \\
&\le C \big( \|p_h^0\|_{L^2(\Omega)} + \|u_h^0\|_{L^2(\Omega)} 
   + \sum\nolimits_{k=1}^{n} \tau ( \|f^{k-1/2}\|_{L^2(\Omega)} + \|g^{k-1/2}\|_{L^2(\Omega)} ) \big).
\end{align*}
\end{lemma}
\begin{proof}
The assertion follows by testing the discrete problem \eqref{eq:syscvfd1}--\eqref{eq:syscvfd2} with $q_h=p_h^{n-1/2}$ and $v_h=u_h^{n-1/2}$, summing over all time-steps, and a discrete Gronwall lemma.
\end{proof}

\subsection{Error analysis}
With similar arguments as for the semi-discretization, we obtain 
\begin{theorem}[Error estimates for the full discretization] \label{theorem:fullydiscrete} $ $\\
Assume that $a \in W^{1,\infty}(\Th)$ and that $(p,u)$ is sufficiently smooth. 
Then 
\begin{align*}
\|\pi_h p(t^n) - p_h^n\|_{L^2(\Omega)} + \|u(t^n) - u_h^n\|_{L^2(\Omega)} 
\le C \big (h^{k+2} C_1(p,u,t^n) + \tau^2 C_3(p,u,t^n) \big)
\end{align*}
with $C_1(p,u,t)$ as in Theorem~\ref{theorem:semidiscrete} and $C_3(p,u,t) = \int_0^t \|\dttt p(s)\|_{L^2(\Omega)} + \|\dttt u(s)\|_{L^2(\Omega)} ds$.
\end{theorem}
\begin{proof}
For ease of notation, we write $p^n=p(t^n)$ and $u^n=u(t^n)$ in the following. 
Similar as for the semi-discrete problem, 
we again use an error splitting 
\begin{align*} 
\|p^n-p_h^n\|_{L^2(\Omega)}  &+ \|u^n-u_h^n\|_{L^2(\Omega)} \\
&\le \|p^n-\pi_h p^n\|_{L^2(\Omega)} + \|u^n-\rho_h u^n\|_{L^2(\Omega)} \notag \\
& \qquad \qquad + \|\pi_h p^n-p_h^n\|_{L^2(\Omega)} + \|\rho_h u^n-u_h^n\|_{L^2(\Omega)} \notag
\end{align*}
into a projection and a discrete error component.
The projection error can be estimated by Lemma~\ref{lem:projections}.
By construction, the discrete errors $r_h^n = \pi_h p(t^{n}) - \p_h^n$ and $w_h^n = \rho_h u(t^n) - u_h^n$ now satisfy $r_h(0)=0$ and $w_h(0)=0$, and 
they solve a fully discrete system of the form \eqref{eq:syscvfd1}--\eqref{eq:syscvfd2} 
with right hand sides
\begin{align*}
f_h^{n-1/2} &= (a-\pi_h^0 a) (\dtau \pi_h p^{n-1/2} - \dtau p^{n-1/2}) +  a (\dtau p^{n-1/2}-\dt p^{n-1/2}) = :(i)+(ii), \\
g_h^{n-1/2} &=  b (\dtau \rho_h u^{n-1/2} - \dtau u^{n-1/2}) + b (\dtau u^{n-1/2} - \dt u^{n-1/2}) =:(i')+(ii').
\end{align*}
The two terms (i) and (i') can be estimated similarly as above by 
\begin{align*}
\|(a-\pi_h^0 a) (\dtau \pi_h p^{n-1/2} - \dtau p^{n-1/2})\|_{L^2(\Omega)}
&\le  \tau^{-1} C h^{k+2} \|a\|_{W^{1,\infty}(\Omega)} \int_{t^{n-1}}^{t^n} \|\dt p(s)\|_{H^{k+1}(\Omega)} ds
\end{align*}
and
\begin{align*}
\|\dtau \rho_h u^{n-1/2} - \dtau u^{n-1/2}\|_{L^2(\Omega)} 
&\le \tau^{-1} C h^{k+2} \int_{t^{n-1}}^{t^n} \|\dt u(s)\|_{H^{k+2}(\Omega)} ds.
\end{align*}
For the remaining terms (ii) and (ii'), we use Taylor expansions to estimate the differences
\begin{align*}
|\dtau z^{n-1/2} - \dt z^{n-1/2}|
&\le \tau \int_{t^{n-1}}^{t^n} |\dttt z(s)| ds.
\end{align*}
Summing up, this yields the following bounds for the right hand sides
\begin{align*}
\tau \|f_h^{n-1/2}\|_{L^2(\Omega)} &\le C \int_{t^{n-1}}^{t^n} h^{k+2} \|\dt p(s)\|_{H^{k+1}\Omega)} + \tau^2  \|\dttt p(s)\|_{L^2(\Omega)} ds,\\
\tau \|g_h^{n-1/2}\|_{L^2(\Omega)} &\le C \int_{t^{n-1}}^{t^n} h^{k+2} \|\dt u(s)\|_{H^{k+2}(\Omega)} + \tau^2  \|\dttt u(s)\|_{L^2(\Omega)} ds.
\end{align*}
Using the discrete energy estimate of Lemma~\ref{lem:fullydiscrete} and the projection error estimates of Lemma~\ref{lem:projections}, we now obtain the remaining bounds for the the discrete error component. 
\end{proof}

\subsection{Auxiliary estimates}

By careful inspection of the previous proof, one can  
obtain the following bound for the error after the first step, 
which will be needed below. 
\begin{lemma} \label{lem:firststep}
Let the assumptions of the previous theorem be valid. Then 
\begin{align*}
&\|\pi_h p^1 - p_h^1\|_{L^2(\Omega)} + \|\rho_h u^1 - u_h^1\|_{L^2(\Omega)} \\
&\qquad \qquad \le C \tau \max_{0 \le s \le \tau} \big( h^{k+2} \|\dt p(s)\|_{H^{k+1}\Omega)} + \tau^2  \|\dttt p(s)\|_{L^2(\Omega)} \\
&\qquad \qquad \qquad \qquad \qquad + h^{k+2} \|\dt u(s)\|_{H^{k+2}(\Omega)} + \tau^2  \|\dttt u(s)\|_{L^2(\Omega)}\big).
\end{align*}
\end{lemma}

We can now prove the following bound for the discrete time derivative of the velocity, which will be required for the analysis of the post-processing scheme below.

\begin{lemma}[Error estimate for discrete time derivatives] $ $ \label{lemma:discretederiv} \\
Let the assumptions of Theorem~\ref{theorem:fullydiscrete} be valid. Then 
\begin{align*}
\|\dt u^{n-1/2} - \dtau u_h^{n-1/2} \|_{L^2(\Omega)} \le C \big (h^{k+1} C_4(p,u,t^n) + \tau^2 C_5(p,u,t^{n+1}) \big),
\end{align*} 
with $C_4(p,u,t)$, $C_5(p,u,t)$ only depending on the solution.
\end{lemma}
\noindent
Recall that $C$ only depends on $\gamma$ and $k$. The precise forms of $C_4(p,u,t)$ and $C_5(p,u,t)$ can be deduced without difficulty from the proof of the theorem.
\begin{proof}
The error can be decomposed into 
\begin{align*}
\|\dt u^{n-1/2} - \dtau u_h^{n-1/2} \|_{L^2(\Omega)}
\le &\|\dt u^{n-1/2} - \dtau u^{n-1/2}\|_{L^2(\Omega)} 
   + \|\dtau u^{n-1/2} - \dtau \rho_h u^{n-1/2}\|_{L^2(\Omega)}\\
 & + \|\dtau \rho_h u^{n-1/2} - \dtau u_h^{n-1/2}\|_{L^2(\Omega)}
=:(i)+(ii)+(iii). 
\end{align*}
The first two terms (i) and (ii) can be estimated similarly as in Theorem~\ref{theorem:fullydiscrete} by
\begin{align*}
(i) &= \|\dt u^{n-1/2} - \dtau u^{n-1/2}\|_{L^2(\Omega)} 
    \le \tau \int_{t^{n-1}}^{t^n} \|\dttt u(s)\|_{L^2(\Omega)} ds,  
\qquad \text{and} \\
(ii) &= \|\dtau u^{n-1/2} - \dtau \rho_h u^{n-1/2}\|_{L^2(\Omega)} 
     \le C \tau^{-1} h^{k+1} \int_{t^{n-1}}^{t^n} \|\dt u(s)\|_{H^{k+1}(\Omega)}.
\end{align*}
To estimate the third contribution, we define 
functions $r_h^{n-1/2} = \dtau \pi_h  p^{n-1/2} - \dtau p_h^{n-1/2}$
and $w_h^{n-1/2} = \dtau \rho_h u^{n-1/2} - \dtau u_h^{n-1/2}$.
Moreover, we denote by 
\begin{align*}
\dtau p^{\tilde n} := \frac{1}{\tau} \big(p^{n+1/2}-p^{n-1/2}\big)
\qquad \text{and} \qquad  
p^{\tilde n} := \frac{1}{2} \big( p^{n+1/2}+p^{n-1/2}\big)
\end{align*}
the differences and averages for a sequence $\{p^{n-1/2}\}_{n \ge 1}$ of functions defined at intermediate time steps $t^{n-1/2}$.
The superscipt $\tilde n$ is used to emphasize that, e.g., $f^{\tilde n}$ is a  combination of values at intermediate time steps.  
The discrete error components $r_h^{n-1/2}$ and $w_h^{n-1/2}$ defined above 
can then be seen to satisfy a discrete variational problem of the form
\begin{alignat*}{2}
(a \dtau r_h^{\tilde n},q_h)_\Omega+(\div\, w_h^{\tilde n} ,q_h)_\Omega &=(f_h^{\tilde n},q_h)_\Omega \qquad &&\forall q_h\in Q_h, \\ 
(b \dtau w_h^{\tilde n},v_h)_\Omega - (r_h^{\tilde n},\div\, v_h)_\Omega &=(g_h^{\tilde n},v_h)_\Omega &&\forall v_h\in V_h 
\end{alignat*}
for all $n \ge 1$, and with initial values given by
\begin{align*}
r_h^{1/2} &= \dtau \pi_h p^{1/2} - \dtau p_h^{1/2} = \tau^{-1} (\pi_h p^1 - p_h^1), \\
w_h^{1/2} &= \dtau \rho_h u^{1/2} - \dtau u_h^{1/2} = \tau^{-1} (\rho_h u^1 - u_h^1),
\end{align*}
and right hand sides defined by
\begin{align*}
f_h^{\tilde n} = a \; (\dtau \dtau \pi_h p^{\tilde n} - \dtau \dt p^{\tilde n})
\qquad \text{and} \qquad 
g_h^{\tilde n} &=  b \; (\dtau \dtau \rho_h u^{\tilde n} - \dtau \dt u^{\tilde n}).
\end{align*}
With similar arguments as used in the proof of Theorem~\ref{theorem:fullydiscrete}, we obtain the bounds
\begin{align*}
\tau \|f_h^{\tilde n}\|_{L^2(\Omega)} &\le C \int_{t^{n-1}}^{t^{n+1}} h^{k+1} \|\dtt p(s)\|_{H^{k+1}\Omega)} + \tau^2  \|\dtttt p(s)\|_{L^2(\Omega)} ds\\
\tau \|g_h^{\tilde n}\|_{L^2(\Omega)} &\le C \int_{t^{n-1}}^{t^{n+1}} h^{k+1} \|\dtt u(s)\|_{H^{k+1}(\Omega)} + \tau^2  \|\dtttt u(s)\|_{L^2(\Omega)} ds,
\end{align*}
and the initial values $r_h^{1/2}$ and $w_h^{1/2}$ can be bounded by means of Lemma~\ref{lem:firststep}.
Via a discrete energy estimate similar to Lemma~\ref{lem:fullydiscrete},
we then obtain bounds for $\|w_h^{\tilde n}\|_{L^2(\Omega)}$ and $\|r_h^{\tilde n}\|_{L^2(\Omega)}$ for all $n \ge 0$. 
This yields the required estimate for the remaining term (iii).
\end{proof}

\begin{remark}
Under appropriate smoothness assumptions, one can again obtain bounds of the form $\|\dt u^{n-1/2} - \dtau u_h^{n-1/2} \|_{L^2(\Omega)} = O(h^{r} + \tau^\sigma)$ for all $r \le k+2$ and $\sigma \le 2$, which
allows us to explain the convergence behavior also for less regular solutions.
\end{remark}

\section{Post-processing for the full discretization} \label{sec:fullpost}

\subsection{The post-processing strategy}

Following the construction on the semi-discrete level, we now propose
the following strategy for post-processing the pressure approximations obtained 
with the fully discrete scheme.%

\begin{problem}[Local post-processing strategy for the full discretization] $ $\\
For all $0 < t^n \le T$ find $\widetilde p_h^{\;n-1/2} \in \P_{k+1}(\Th)$ such that
\begin{alignat}{2}
(\grad\widetilde p_h^{\;n-1/2},\grad \widetilde q_h)_K&=
(g^{n-1/2},\grad \widetilde q_h)_K-(b\dtau\u_h^{n-1/2},\grad \widetilde q_h)_K\quad &&\forall \widetilde q_h\in \P_{k+1}(K) \label{eq:sysppd1} \\ 
(\widetilde p_h^{\;n-1/2},q_h^0)_K&=(p_h^{n-1/2},q_h^0)_K &&\forall q_h^0\in \P_0(K) \label{eq:sysppd2},
\end{alignat}
\end{problem}
Note that $\widetilde p_h^{\;n-1/2}$ can be be constructed independently on every element $K \in T_h$ and for every time step $t^{n-1/2}$, which makes this procedure computationally attractive.

\subsection{Analysis of the post-processing scheme}
As the next theorem shows, the improved pressure approximations again exhibit super-convergence.
\begin{theorem}\label{theorem:post-processingfull}
Assume that $(p,u)$ is sufficiently smooth. Then 
\begin{align*}
\|p^{n-1/2}-\widetilde p_h^{\;n-1/2}\|_{L^2(\Omega)}
\leq C \big( h^{k+2} C_6(p,u,t^{n}) + \tau^2 C_7(p,u,t^{n+1})\big)
\end{align*}
with $C_6(p,u,t^n)$ and $C_7(p,u,t^n)$ depending only on the solution $(p,u)$ on  $[0,t^n]$. 
\end{theorem}
The precise form of $C_6(p,u,t)$ and $C_7(p,u,t)$ can again be deduced 
without difficulty from the proof and the previous results.
\begin{proof}
Using exactly the same arguments as in the proof of Theorem~\ref{theorem:semidiscrete-pp}, we obtain
\begin{align*}
\|p^{n-1/2}&-\widetilde p_h^{\;n-1/2}\|_{L^2(\Omega)}
\leq \|p^{n-1/2}-\widetilde \pi_h p^{n-1/2}\|_{L^2(\Omega)} + 
\|\pi_h p^{n-1/2}-p_h^{n-1/2}\|_{L^2(\Omega)}\\
& + C\,h \; \left(\|\grad(\widetilde \pi_h p^{n-1/2}-p^{n-1/2})\|_{L^2(\Omega)} + 
\|\dt u^{n-1/2}-\dtau u_h^{n-1/2})\|_{L^2(\Omega)}\right)
\end{align*}
The first and third term can be bounded by the projection error estimates.
The second can be estimated as in Theorem~\ref{theorem:fullydiscrete}, and the last one can be bounded by Lemma~\ref{lemma:discretederiv}.
\end{proof}

\begin{remark}
A careful inspection of the proof shows that it would be sufficent to guarantee that 
$h \|\dt u^{n-1/2}-\dtau u_h^{n-1/2}\|_{L^2(\Omega)} = O(h^{k+2} + \tau^2)$.
If $h \le C \tau$, which is a reasonable assumption, the estimate $\|\dt u^{n-1/2}-\dtau u_h^{n-1/2}\|_{L^2(\Omega)}=O(h^{k+1} + \tau)$ 
would therefore be sufficient. The smoothness requirements in Lemma~\ref{lemma:discretederiv} could then be further relaxed.
\end{remark}

\section{Numerical tests} \label{sec:num}
We now illustrate our theoretical results by some numerical tests. 
In order to demonstrate that the results hold without regularity of 
the adjoint problem, we consider as computational domain the L-shape $\Omega=(-1,1)^2\setminus[0,1]^2$, which is non-convex.
For all tests, the model parameters are set to $a=2,b=1$, and 
we choose the right hand sides to be $f=g=0$. The time horizon is set to $T=1$.
We only report about results for $\BDM_1$--$\P_0$ elements,
which corresponds to the lowest order case $k=0$ in the theorems. 

\subsection{Smooth solution}
In the first test, we define the initial values to be
\begin{align*}
u_0(x,y)=[0;0]
\quad \text{and} \quad 
p_0(x,y)=\sin(\pi x) \sin(\pi y).
\end{align*}
The solution of problem \eqref{eq:sysm1}--\eqref{eq:sysm4} 
with the above problem data is then given by 
\begin{align*}
p(x,y,t)&= \sin(\pi x) \sin(\pi y) \cos(\pi t),\\
u(x,y,t)&=-[\cos(\pi x) \sin(\pi y);\sin(\pi x)\cos(\pi y)] \sin(\pi t).
\end{align*}
Since the solution is infinitely differentiable here, 
we expect to obtain the full convergence of order $O(h^2+\tau^2)$.
For our first series of tests, we used a small time step of $\tau=1/1000$, 
such that the effects of time discretization can be neglected.
As error measure, we utilize the norm $\tnorm e\tnorm:=\max_{0 \le t^n \le T} \|e(t^n)\|_{L^2(\Omega)}$. 
To simplify the computations, the exact solution is approximated by piecewise constant and piecewise linear functions obtained by $L^2$-projections $\pi^0_h$ and $\pi_h^1$.

In Table~\ref{tab:1} we list the results obtained for a sequence of uniformly refined meshes. 
\begin{table}[ht!]
\begin{tabular}{c||c|c||c|c||c|c} 
$h$ & $\tnorm \pi^1_h u - u_h\tnorm$ & eoc & $\tnorm\pi^0_h p - p_h\tnorm$ & eoc & $\tnorm \pi^1_h p - \widetilde p_h \tnorm$ &eoc \\
\hline
\hline
$2^{-2}$ & $0.492260$ & ---    & $0.444531$ & ---    & $0.492531$ & ---    \\
$2^{-3}$ & $0.142499$ & $1.79$ & $0.136521$ & $1.70$ & $0.144388$ & $1.77$ \\
$2^{-4}$ & $0.037066$ & $1.94$ & $0.036047$ & $1.92$ & $0.037627$ & $1.94$ \\
$2^{-5}$ & $0.009359$ & $1.99$ & $0.009128$ & $1.98$ & $0.009499$ & $1.99$ \\
$2^{-6}$ & $0.002347$ & $2.00$ & $0.002290$ & $1.99$ & $0.002380$ & $2.00$ \\
$2^{-7}$ & $0.000586$ & $2.00$ & $0.000572$ & $2.00$ & $0.000594$ & $2.00$ 
\end{tabular}
\medskip
\caption{Errors vs mesh size $h$ obtained with the $\BDM_1$--$\P_0$ method on the L-shape geometry with smooth exact solution and time step $\tau=1/1000$.\label{tab:1}} 
\end{table}
As predicted by Theorem~\ref{theorem:semidiscrete},
we observe the optimal second order convergence for the velocity error and the super-convergence for the projected pressure error. 
By the post-processing procedure, we obtain a piecewise linear approximation $\widetilde p_h$ for the pressure, which also shows second order convergence as stated in Theorem~\ref{theorem:semidiscrete-pp}.

In Table~\ref{tab:1a}, we list the corresponding results obtained on the finest mesh with mesh size $h=2^{-7}$ and for varying size of the time step $\tau$. 
\begin{table}[ht!]
\begin{tabular}{c||c|c||c|c||c|c} 
$\tau$ & $\tnorm \pi^1_h u - u_h\tnorm$ & eoc & $\tnorm\pi^0_h p - p_h\tnorm$ & eoc & $\tnorm \pi^1_h p - \widetilde p_h \tnorm$ & eoc \\
\hline
\hline
$2^{-2}$ & $0.180514$ & ---    & $0.071484$ & ---    & $0.101313$ & ---    \\
$2^{-3}$ & $0.048197$ & $1.90$ & $0.019738$ & $1.86$ & $0.028855$ & $1.81$ \\
$2^{-4}$ & $0.012163$ & $1.99$ & $0.004928$ & $2.00$ & $0.007376$ & $1.97$ \\
$2^{-5}$ & $0.002958$ & $2.04$ & $0.001164$ & $2.08$ & $0.001762$ & $2.07$ \\
$2^{-6}$ & $0.000645$ & $2.20$ & $0.000231$ & $2.33$ & $0.000348$ & $2.34$ \\
$2^{-7}$ & $0.000014$ & $2.20$ & $0.000129$ & $0.83$ & $0.000009$ & $1.96$ 
\end{tabular}
\medskip
\caption{Errors vs time step $\tau$ obtained with the $\BDM_1$--$\P_0$ method on the L-shape geometry with smooth exact solution and mesh size $h=2^{-7}$.\label{tab:1a}} 
\end{table}
As predicted by Theorems~\ref{theorem:fullydiscrete} and \ref{theorem:post-processingfull}, we also obtain second order convergence with respect to time step $\tau$. 
A comparison with Table~\ref{tab:1} shows that for the smallest time step, the discretization error due to spatial discretization already starts to dominate, which explains the saturation in the second column.  
Overall, the numerical results for the case of a smooth solution are in very good  agreement with the theoretical predictions.

\subsection{Non-smooth solution}

As a second test case, we consider the same geometry and parameters as in the previous example, but as initial conditions we now choose
\begin{align*}
u_0(x,y)=[0;0]
\quad \text{and} \quad p_0(x,y)=\begin{cases} \sin(\pi x) \sin(\pi y), & -1 \le x,y, \le 0, \\ 0, & \text{else}.\end{cases}
\end{align*}
Note that $p_0$ has a kink and therefore $p(0) \not\in H^{3/2}(\Omega)$ but $p(0) \in H^{3/2-\eps}(\Omega)$ for any $\eps>0$. As a consequence, we have $\dt u(0) = -\nabla p(0) \in H^{1/2-\eps}(\Omega)^2$ but $\dt u(0) \not\in H^{1/2}(\Omega)^2$. 
The same regularity of the functions and the time derivatives is 
preserved for $t>0$. 

Following our remarks after Theorems~\ref{theorem:fullydiscrete} and \ref{theorem:post-processingfull}, we would expect here to obtain at least convergence of order $h^{1/2-\eps}$ for the error in the velocity, the projected pressure error, and for the post-processed pressure. 
Since no analytical solution is available here, we utilize the 
distance to the solution on the next finer mesh to evaluate the error.
The results of our numerical experiments are listed in Tables~\ref{tab:2} and \ref{tab:2a}.
\begin{table}[ht!]
\begin{tabular}{c||c|c||c|c||c|c} 
$h$ & $\tnorm u_{h/2} - u_h\tnorm$ & eoc & $\tnorm\pi_h p_{h/2} - p_h\tnorm$ & eoc & $\tnorm\pi^1_h p_{h/2} - \widetilde p_h \tnorm$ &eoc \\
\hline
\hline
$2^{-2}$ & $0.188661$ & ---    & $0.107367$ & ---    & $0.136479$ & ---    \\
$2^{-3}$ & $0.086160$ & $1.13$ & $0.037058$ & $1.53$ & $0.047367$ & $1.53$ \\
$2^{-4}$ & $0.035437$ & $1.28$ & $0.016544$ & $1.16$ & $0.019431$ & $1.29$ \\
$2^{-5}$ & $0.018877$ & $0.91$ & $0.008794$ & $0.91$ & $0.009787$ & $0.99$ \\
$2^{-6}$ & $0.010418$ & $0.86$ & $0.004480$ & $0.97$ & $0.004790$ & $1.03$ 
\end{tabular}
\medskip
\caption{Errors vs mesh size $h$ obtained with the $\BDM_1$--$\P_0$ method on the L-shape geometry with non-smooth solution and time step $\tau=1/1000$.\label{tab:2}} 
\end{table}

We here observe convergence of approximately order one with respect to the mesh size, 
which is actually better than guaranteed by Theorem~\ref{theorem:semidiscrete} and \ref{theorem:semidiscrete-pp}. 
To evaluate the convergence with respect to the time step size, we 
repeat the computations on the finest mesh with mesh size $h=2^{-7}$ and for different values of $\tau$. The error is here computed by comparison with the numerical solution obtained with time step $\tau/2$. 
The corresponding results are depicted in Table~\ref{tab:2a}.
\begin{table}[ht!]
\begin{tabular}{c||c|c||c|c||c|c} 
$\tau$ & $\tnorm u_h^{\tau/2} - u_h\tnorm$ & eoc & $\tnorm p_h^{\tau/2} - p_h\tnorm$ & eoc & $\tnorm  \widetilde p_h^{\;\tau/2} - \widetilde p_h \tnorm$ & eoc \\
\hline
\hline
$2^{-2}$ & $0.071837$ & ---    & $0.043523$ & ---    & $0.059628$ &  ---    \\
$2^{-3}$ & $0.034591$ & $1.05$ & $0.022490$ & $0.95$ & $0.022620$ &  $1.40$ \\
$2^{-4}$ & $0.016427$ & $1.07$ & $0.011389$ & $0.98$ & $0.010293$ &  $1.14$ \\
$2^{-5}$ & $0.008140$ & $1.01$ & $0.005732$ & $0.99$ & $0.005231$ &  $0.97$ \\
$2^{-6}$ & $0.004081$ & $1.00$ & $0.002883$ & $0.99$ & $0.002713$ &  $0.95$ 
\end{tabular}
\medskip
\caption{Errors vs time step $\tau$ obtained with the $\BDM_1$--$\P_0$ method on the L-shape geometry with non-smooth solution and mesh size $h=2^{-7}$.\label{tab:2a}} 
\end{table}

Again we observe convergence of approximately order one also with respect to the time step size, which is better than actually predicted by Theorems~\ref{theorem:fullydiscrete} and \ref{theorem:post-processingfull}.
The numerical results thus indicate that the smoothness requirements in our theorems could possibly be further relaxed. 
Let us note that a similar sub-optimality with respect to the smoothness requirements can also be found in  \cite{Baker76,BakerBramble79,BangerthRannacher99,Chen96,Chen99,CowsarDupontWheeler96,DouglasDupontWheeler78,Dupont73,Geveci88,JenkinsRiviereWheeler02,Makridakis92,MakridakisMonk95,Monk92}.
A full explanation of the convergence rates for finite element approximations of hyperbolic problems under mimimal smoothness assumptions therefore seems to be an open problem.

\section{Discussion}

In this paper we considered the numerical approximation of the acoustic wave equation by a mixed finite element method based on $\BDM_{k+1}$-$\P_k$ elements. 
Optimal convergence for the error in the velocity and super-convergence of the projected pressure order was established. 
Based on these results, we proposed a post-processing strategy that allows to 
obtain an improved approximation for the pressure in $\P_{k+1}$ by purely local computations.
This shows that the $\BDM_{k+1}$-$\P_k$ discretization is in fact a balanced 
approximation, although different polynomial orders are used.
Since no duality argument nor inverse inequalities are used in our analysis, 
the results hold for general domains and for unstructured as well as locally refined meshes.

The a-priori error estimates and the post-processing strategy were extended to fully discrete schemes obtained by time discretization via the Crank-Nicolson method. The generalization of the arguments to time discretizations of higher order is possible.

In view of numerical efficiency, it might be desirable to consider explicit time integration schemes together with some appropriate mass lumping strategy \cite{Cohen02}. 
Numerical experiments indicate that the super-convergence of the projected pressure error remains valid also in that case. The error of the velocity approximation however is reduced by one order as effect of the consistency errors introduced by the mass-lumping. 
It is not clear, if the optimal approximation order in both variables can be restored by another type of post-processing, and we leave this as an open topic for future research.

Our numerical results also indicate 
that the regularity requirements in the error estimates could possibly be further relaxed. 
This applies to all estimates for mixed finite element approximations of wave propagation problems, we are aware of. It would thus be desirable to further relax the smoothness requirements in order to be able to fully explain the convergence behavior also in the case of non-smooth solutions.

\section*{Acknowledgements}

The authors are grateful for financial support by the German Research Foundation (DFG) via grants IRTG~1529 and TRR~154,
and by the ``Excellence Initiative'' of the German Federal and State Governments via the Graduate School of Computational Engineering GSC~233 at Technische Universität Darmstadt.


\begin{thebibliography}{10}

\bibitem{ArnoldBrezzi85}
D.~Arnold and F.~Brezzi.
\newblock Mixed and nonconforming finite element methods: Implementation,
  postprocessing and error estimates.
\newblock {\em RAIRO Model. Math. Anal. Numer.}, 19:7--32, 1985.

\bibitem{Baker76}
G.~A. Baker.
\newblock Error estimates for finite element approximations for second order
  hyperbolic equations.
\newblock {\em SIAM J. Numer. Anal.}, 13:564--576, 1976.

\bibitem{BakerBramble79}
G.~A. Baker and J.~H. Bramble.
\newblock Semidiscrete and single step fully discrete approximations for second
  order hyperbolic equations.
\newblock {\em RAIRO Anal. Numer.}, 13:75--100, 1979.

\bibitem{BangerthRannacher99}
W.~Bangerth and R.~Rannacher.
\newblock Adaptive finite element techniques for the acoustic wave equation.
\newblock {\em J.~Comput.~Acoustics}, 1:1--17, 1999.

\bibitem{BoffiEtAl08}
D.~Boffi, F.~Brezzi, L.~F. Demkowicz, R.~G. Dur{\'a}n, R.~S. Falk, and
  M.~Fortin.
\newblock {\em Mixed finite elements, compatibility conditions, and
  applications}, volume 1939 of {\em Lecture Notes in Mathematics}.
\newblock Springer-Verlag, Berlin; Fondazione C.I.M.E., Florence, 2008.

\bibitem{BoffiBrezziFortin13}
D.~Boffi, F.~Brezzi, and M.~Fortin.
\newblock {\em Mixed finite element methods and applications}, volume~44 of
  {\em Springer Series in Computational Mathematics}.
\newblock Springer, Heidelberg, 2013.

\bibitem{BDDF87}
F.~Brezzi, J.~Douglas, R.~Dur{\`a}n, and M.~Fortin.
\newblock Mixed finite elements for second order elliptic problems in three
  variables.
\newblock {\em Numer. Math.}, 47:237--250, 1987.

\bibitem{BrezziDouglasMarini85}
F.~Brezzi, J.~Douglas, and L.~D. Marini.
\newblock Two families of mixed elements for second order elliptic problems.
\newblock {\em Numer. Math.}, 88:217--235, 1985.

\bibitem{Chen99}
Y.~Chen.
\newblock Global superconvergence for a mixed finite element method for the
  wave equation.
\newblock {\em Systems Sci. Math. Sci.}, 12:159--165, 1999.

\bibitem{ChenHuang98}
Y.~Chen and Y.~Huan.
\newblock The superconvergence of mixed finite element methods for nonlinear
  hyperbolic equations.
\newblock {\em Numer. Simul.}, 3:155--158, 1998.

\bibitem{Chen96}
Z.~Chen.
\newblock Superconvergence results for {G}alerkin methods for wave propagation
  in various porous media.
\newblock {\em Numer. Meth. Part. Diff. Equat.}, 12:99--122, 1999.

\bibitem{Cohen02}
G.~Cohen.
\newblock {\em Higher-Order Numerical Methods for Transient Wave Equations}.
\newblock Springer, Heidelberg, 2002.

\bibitem{CohenJolyRobertsTordjman01}
G.~Cohen, O.~Joly, J.~E. Roberts, and N.~Tordjman.
\newblock Higher order triangular finite elements with mass lumping for the
  wave equation.
\newblock {\em SIAM J. Numer. Anal.}, 38:2047--2078, 2001.

\bibitem{CowsarDupontWheeler96}
L.~C. Cowsar, T.~F. Dupont, and M.~F. Wheeler.
\newblock A priori estimates for mixed finite element approximations of
  second-order hyperbolic equations with absorbing boundary conditions.
\newblock {\em SIAM J. Numer. Anal.}, 33:492--504, 1996.

\bibitem{DouglasDupontWheeler78}
J.~Douglas, T.~Dupont, and M.~F. Wheeler.
\newblock A quasi-projection analysis of {G}alerkin methods for parabolic and
  hyperbolic equations.
\newblock {\em Math. Comp.}, 32:345--362, 1978.

\bibitem{Dupont73}
T.~Dupont.
\newblock $l^2$ estimates for {G}alerkin methods for second-order hyperbolic
  equations.
\newblock {\em SIAM J. Numer. Anal.}, 10:880--889, 1973.

\bibitem{Evans98}
L.~C. Evans.
\newblock {\em Partial Differential Equations}, volume~19 of {\em Graduate
  Studies in Mathematics}.
\newblock American Mathematical Society, 1998.

\bibitem{Geveci88}
T.~Geveci.
\newblock On the application of mixed finite element methods to the wave
  equations.
\newblock {\em RAIRO Model. Math. Anal. Numer.}, 22:243--250, 1988.

\bibitem{JenkinsRiviereWheeler02}
E.~W. Jenkins, T.~Rivi{\`e}re, and M.~F. Wheeler.
\newblock A priori error estimates for mixed finite element approximations of
  the acoustic wave equation.
\newblock {\em SIAM J. Numer. Anal.}, 40:1698--1715, 2002.

\bibitem{Joly03}
P.~Joly.
\newblock Variational methods for time-dependent wave propagation problems.
\newblock In {\em Topics in Computational Wave Propagation}, volume~31 of {\em
  LNCSE}, pages 201--264. Springer.

\bibitem{Karaa11}
S.~Karaa.
\newblock Error estimates for finite element approximations of a viscous wave
  equation.
\newblock {\em Numer. Func. Anal. Optim.}, 32:750--767, 2011.

\bibitem{LandauLifshitz6}
L.~D. Landau and E.~M. Lifshitz.
\newblock {\em Course of theoretical physics. {V}ol. 6}.
\newblock Pergamon Press, Oxford, second edition, 1987.
\newblock Fluid mechanics, Translated from the third Russian edition by J. B.
  Sykes and W. H. Reid.

\bibitem{Makridakis92}
C.~G. Makridakis.
\newblock On mixed finite element methods for linear elastodynamics.
\newblock {\em Numer. Math.}, 61:235--260, 1992.

\bibitem{MakridakisMonk95}
C.~G. Makridakis and P.~Monk.
\newblock Time-discrete finite element schemes for {M}axwell's equations.
\newblock {\em RAIRO Model. Math. Anal. Numer.}, 29:171--197, 1995.

\bibitem{Monk92}
P.~Monk.
\newblock Analysis of a finite element methods for {M}axwell's equations.
\newblock {\em SIAM J. Numer. Anal.}, 29:714--729, 1992.

\bibitem{Pazy83}
A.~Pazy.
\newblock {\em Semigroups of linear operators and applications to partial
  differential equations}, volume~44 of {\em Applied Mathematical Sciences}.
\newblock Springer-Verlag, New York, 1983.

\bibitem{Schoeberl01}
J.~Sch\"oberl.
\newblock Commuting quasi-interpolation operators for mixed finite elements.
\newblock Isc-01-10-math, Institute for Scientific Computing, Texas A\&M
  University, 2001.

\bibitem{Stenberg91}
R.~Stenberg.
\newblock Postprocessing schemes for some mixed finite elements.
\newblock {\em RAIRO Model. Math. Anal. Numer.}, 25:151--167, 1991.

\bibitem{WangChenTang13}
F.~Wang, Y.~Chen, and Y.~Tang.
\newblock Superconvergence of fully discrete splitting positive definite mixed
  {FEM} for hyperbolic equations.
\newblock {\em Numer. Meth. Part. Diff. Equat.}, 30:175--186, 2014.

\bibitem{WangBathe97}
X.~Wang and K.-J. Bathe.
\newblock Displacement/pressure based mixed finite element formulations for
  acoustic fluid-structure interaction problems.
\newblock {\em Int. J. Numer. Meth. Engrg.}, 40:2001--2017, 1997.

\end{thebibliography}

\end{document}